\documentclass[11pt]{amsart}
\author{Maya Saran}
\subjclass[2000]{03E15, 28A05, 54H05}
\keywords{descriptive set theory, ideals of compact sets}

\address{Department of Mathematics\\
1409 W. Green St.\\
University of Illinois\\
Urbana, IL 61801, USA}

\email{msaran@math.uiuc.edu}
\title{A note on $\gd$ ideals of compact sets}
\usepackage{amsmath}
\usepackage{amsthm}
\usepackage{amssymb,latexsym}

\newtheorem{thm}{Theorem}

\newtheorem{cor}[thm]{Corollary}

\newcommand{\F}{\mathcal{F}}
\newcommand{\D}{\mathcal{D}}

\newcommand{\el}{\mathcal{L}}

\newcommand{\U}{\mathcal{U}}
\newcommand{\V}{\mathcal{V}}

\newcommand{\K}{\mathcal{K}}
\newcommand{\gd}{G_{\delta}}

\newcommand{\ol}[1]{\overline{#1}}
\newcommand{\bbn}{\mathbb{N}}

\begin{document}

\begin{abstract}
Solecki has shown that a broad natural class of $\gd$ ideals of compact sets can be represented through the ideal of nowhere dense subsets of a closed subset of the hyperspace of compact sets. In this note we show that the closed subset in this representation can be taken to be closed upwards. 
\end{abstract}

\maketitle

Let $E$ be a compact Polish space and let $\K(E)$ denote the hyperspace of its compact subsets, equipped with the Vietoris topology. A set $I \subseteq \K(E)$ is an  \emph{ideal} of compact sets if it is closed under the operations of taking subsets and finite unions. An ideal $I$ is a \emph{$\sigma$-ideal} if it is also closed under countable unions whenever the union itself is compact. Ideals of compact sets arise commonly in analysis out of various notions of smallness; see \cite{MZ} for a survey of results and applications.

After  \cite{SS1} we say that an ideal $I$ has \emph{property $(*)$}  if, for any sequence of sets $K_n \in I$, there exists a $\gd$ set $G$ such that $\bigcup_n K_n \subseteq G$ and $\K(G) \subseteq I$. Property $(*)$ holds in a broad class of $\gd$ ideals that includes all natural examples, including the ideals of compact meager sets, measure-zero sets, sets of dimension $\leq n$ for fixed $n \in \bbn$, and Z-sets. (See \cite{SS1} for these and other examples and a discussion of property $(*)$.) Solecki has shown in \cite{SS1} that ideals in this class are represented via the meager ideal of a closed subset of $\K(E)$. The following definition is essential to the representation: for $A \subseteq E$, $A^*= \{K \in \K (E) : K \cap A \neq \emptyset \}$.

\begin{thm}[Solecki]\label{thm:sschar}
Suppose $I$ is coanalytic and non-empty. Then $I$ has property $(*)$ iff there exists a closed set $\F \subseteq \K(E)$ such that, for any $K \in \K(E)$, \[K \in I \iff K^* \cap \F \textrm{ is meager in } \F.\]
\end{thm}

This representation is analogous to a result of Choquet \cite{Choquet} that establishes a correspondence between alternating capacities of order $\infty$ on $E$ and probability Borel measures on $\K(E)$.

Note that the set $\F$ in  Theorem~\ref{thm:sschar} is not unique. We hope to determine properties for $\F$ that make it a canonical representative, perhaps upto some notion of equivalence. One property of interest is that of being \emph{closed upwards}, i.e., $\forall A,B \in \K(E), \; B \supseteq A \in \F \Rightarrow B \in \F$. This property ensures that the map $K \mapsto K^* \cap \F$, a fundamental function in this context, is continuous.  In several examples of $\gd$ ideals with property $(*)$, the natural choice of the set $\F$ is in fact closed upwards. For example, let $\mu$ be an atomless finite probability measure on $E$ and let $I$ be the $\sigma$-ideal of compact $\mu$-null sets. Fix a basis of the topology on $E$ and let $s \in (0,1)$ be chosen so that it is not the measure of any basic set. Then the set  $\F = \{K \in \K(E) : \mu (K) \geq s \}$ works to characterize membership in the ideal. 

In the following result we show that as long as the ideal $I$ in Theorem~\ref{thm:sschar} contains only meager sets, we may always find an $\F$ representing it that is closed upwards. We use the following notation in the proof: if $A\subseteq E$ and $\delta > 0$, $A + \delta$ denotes the set $\bigcup_{x \in A}B(x,\delta).$ $Int(A)$ denotes the interior of $A$ in $E$. 

\begin{thm}\label{thm:upclosed}
For a closed set $\F \subseteq \K(E)$, the following are equivalent:
\begin{enumerate}
\item[(1)] $\forall K \in \K(E), \; K$ has nonempty interior $\Rightarrow K^* $ nonmeager in $\F$.
\item[(2)] $\exists \F' \subseteq \K(E)$, closed and closed upwards, such that  \[ \forall K \in \K(E), K^* \textrm{ nonmeager in }\F'\iff K^* \textrm{ nonmeager in }\F.\]
\end{enumerate}
\end{thm}

\begin{proof}

It is clear that (2) $\Rightarrow$ (1), simply because, if $\F'\subseteq \K(E)$ is closed upwards and $U \subseteq E$ is open, $\F' \cap U^*$ is open and nonempty. To prove the other direction, let $I = \{K \in \K(E): K^* $ is meager in $ \F\}$. $I$ is an ideal with property $(*)$. Let $\{\V_n\}$ be a basis for the relative topology on $\F$, and let $\K_n = \ol{\V_n}$. We now have:
\[
K \in I \Rightarrow \forall n \; K^* \textrm{ meager in } \K_n \textrm{ and }K \notin I \Rightarrow \exists n \; \K_n \subseteq K^* 
\]

Fix a sequence $\{x_i\}$ and a point $x \in E$ such that the $x_i$ are all distinct, $x_i \to x$, $\{x\} \in I$ and each $\{x_i\} \in I$. (This is easy -- we can just pick the $x_i$ from some fixed infinite set in $I$.) Let  $U'_i$ be open such that $x_i \in U'_i$, $\ol{U'_i} \to \{x\}$ and the $\ol{U'_i}$ are pairwise disjoint. Now we pick a subsequence $U'_{n_i}$ and define sets $(U_i, F_i, W_i)$, $i \in \mathbb{N}$, satisfying:
\begin{itemize}
\item $U_i, W_i$ are open,
\item $U_i \subseteq U'_{n_i}$, so the sets $\ol{U_i}$ are pairwise disjoint,
\item $F_i \in \K_i$,
\item $F_i \subseteq W_i$,
\item If $j \leq i$ then $ \ol{W_j} \cap \ol{U_i} = \emptyset$.
\end{itemize}
Since $\{x, x_0\} \in I, \; \K_0 \nsubseteq \{x, x_0\}^*$. Let $F_0 \in \K_0$ such that $x \notin F_0$, $x_0 \notin F_0$. Let $W_0$ be some open set such that $F_0 \subseteq W_0$ and $x, x_0 \notin \ol{W_0}.$ Let $U_0 \subseteq U'_0$ be an open set containing $x_0$ such that $\ol{U_0} \cap \ol{W_0} = \emptyset$. 

Pick $n_1 > 0$ such that $\forall m \ge n_1, \; \ol{W_0} \cap \ol{U'_m} = \emptyset.$  

To define $(U_i, F_i, W_i)$, consider $\K_i$ and $U'_{n_i}$. Pick $F_i \in \K_i \setminus \{x, x_{n_i}\}^*.$ Let $W_i \supseteq F_i$ be open such that $x, x_{n_i} \notin \ol{W_i}.$ Let $U_i \subseteq U'_{n_i}$ be an open set containing $x_{n_i}$ such that $\ol{U_i} \cap \ol{W_i} = \emptyset$. Pick  $ n_{i+1} > n_i $ such that $\forall m \ge n_{i+1}, \; \ol{W_i} \cap \ol{U'_m} = \emptyset.$ 

Note that:
\[
K \in I \Rightarrow \forall n \; K^* \textrm{ meager in } \K_n \cap \K(W_n) \textrm{ and }K \notin I \Rightarrow \exists n \; \K_n \cap \K(W_n) \subseteq K^* 
\]
So, by replacing the sets $\K_n$ by the sets $\K_n \cap \K(W_n)$, we may simply assume that $\K_n \subseteq \K(W_n)$.\\

Now define $\el \subseteq \K(E)$ as follows. Fix $n \in \mathbb{N}.$ For $j \in \mathbb{N}$, define closed sets

\begin{displaymath}
A_{n,j}= \left\{ \begin{array}{ll}
\ol{U_j}& \textrm{if } j<n \\
E\setminus \bigcup_{i<n}(U_i + 1/j) & \textrm{if } j \ge n
\end{array} \right.
\end{displaymath}

Let $U_{n,j}, \; j \in \mathbb{N},$ be nonempty disjoint open subsets of $U_n$. (This is possible because, since $\{x_n\}$ is not open, it must lie in the  perfect part of $E$.)\\

Define sets $\el_{n,j}$ as follows: for  $L \in \K(E)$,
\[L \in \el_{n,j} \iff \exists F \in \K_n  \; \textrm{ such that } F \cap A_{n,j} \subseteq L \textrm{ \emph{and} } L \textrm{ intersects } U_{n,j}\]

Now let $\el = \bigcup_{n,j} \el_{n,j}.$ Since each $\el_{n,j}$ is closed upwards, so is $\el$. \\

Claim: $K \in I \iff K^* $ is meager in $ \el.$\\

Let $K \in I$. We want to show that $\el \setminus K^*$ is dense in $\el.$ Let $L_1 \in \el_{n,j}$, i.e., $\exists F \in \K_n  \; \textrm{ such that } F \cap A_{n,j} \subseteq L \textrm{ and } L \textrm{ intersects } U_{n,j}.$ Let $L \supseteq L_1$ be close to $L_1$, satisfying $L_1 \subseteq Int(L)$ and $\ol{Int(L)} = L$. Note that $L$ is nonmeager in $U_{n,j}.$\\

Consider the set $\D = \K_n \cap \{ F: F \cap A_{n,j} \subseteq Int(L) \}.$ $\D$ is a nonempty open subset of $\K_n$. (Openness follows from this easily checked fact about $\K(E)$: if $A\subseteq E$ is closed and $U\subseteq E$ is open, then $\{F \in \K(E): F \cap A \subseteq U\}$ is open.) Since $K \in I$, $K^*$ is meager in $\K_n$. So $\D \nsubseteq K^*$. Let $F_1 \in \D \setminus K^*.$ Now we can remove from $L$ an open $U \supseteq K$ where $U$ is chosen small enough so that $U \cap F_1 = \emptyset$ and $L \setminus U$ is still nonmeager in $U_{n,j}$. Now, $L \setminus U$ is in $\el_n \setminus K^*$ and is close to $L$. \\

Conversely, suppose $K \notin I.$ We want to show that $\exists$ open $\U \subseteq \K(E)$ such that $\emptyset \neq \U \cap \el \subseteq K^*.$\\

Let $C = \bigcup_n \ol{U_n} \cup \{x\}$, a closed set. Write $K \setminus C = \bigcup_j K_j$, where $K_j = K \setminus (C + 1/j)$, which is closed.  Now, \[K = (K \cap \{x\}) \; \cup \; \bigcup_n (K \cap \ol{U_n}) \; \cup \; \bigcup_j K_j \]

Since $I$ is a $\sigma$-ideal and $\{x\} \in I$, we have two possible cases: either some $K \cap \ol{U_n} \notin I$ or some $K_j \notin I.$\\

Case 1: $\exists n \; K \cap \ol{U_n} \notin I$. Fix $m$ such that $\K_m \subseteq (K \cap \ol{U_n})^*$.  

If $m \leq n$ then $\ol{U_n} \cap \ol{W_m} = \emptyset.$ So $m>n$. This means that $\ol{U_n}$ is one of the sets $A_{m,j}. $ Let $V \supseteq \ol{U_n}$ be open such that $V \cap \ol{U_i} = \emptyset \; \forall i \neq n$ and $V \cap \ol{W_n} = \emptyset$. Let $W = V \cup U_{m,j}$. \\

Claim: $\emptyset \neq \el \cap \K(W) \subseteq K^*.$\\

It is clear that $\el_{m,j} \cap \K(W)  \neq \emptyset.$ Let $L \in \K(W) \cap \el$. For any $i \notin \{n,m\}$, $L \cap U_i = \emptyset$. Also, $L \cap W_n = \emptyset$ and $L \cap U_{m,j'} = \emptyset \; \forall j' \neq j$. So the only possibility is that $L \in \el_{m,j},$ i.e., $\exists F \in \K_m \textrm{ such that } F\cap A_{m,j}= F\cap \ol{U_n} \subseteq L$. Since $F \cap \ol{U_n} \cap K \neq \emptyset$, we have $L \cap K \neq \emptyset.$\\

Case 2: $ \exists j \; K_j \notin I.$ Fix $m$ such that $\K_m \subseteq K_j^*$. Fix $\delta > 0$ such that $K_j \cap \bigcup_{i<m} \ol{(U_i + \delta)} = \emptyset$ and let $k \in \mathbb{N}$ such that $k \ge m$ and $1/k < \delta$. Let $ W = (W_m \setminus \bigcup_{i<m} \ol{U_i}) \; \cup \; U_{m,k}.$\\

Claim:  $\emptyset \neq \el \cap \K(W) \subseteq \K^*.$\\ 

It is clear that $\K(W) \cap \el_{m,k} \neq \emptyset$ (To get something in this set, we can simply take any $F \in \K_m$ and join some piece of $U_{m,k}$ to $F \cap A_{m,k}.$) So $K(W) \cap \el \neq \emptyset.$

Now, let $L \in K(W) \cap \el.$ As before, the only possibility is that $L \in \el_{m,k}$, i.e., $\exists F \in \K_m \textrm{ such that } F \cap A_{m,k} = F \setminus \bigcup_{i<m} {(U_i + 1/k)} \subseteq L$. Since $F \in \K_m, \; F \cap K_j \neq \emptyset.$ Let $x \in F \cap K_j$. Since $1/k < \delta,$ we have $x \in L$. So $L \in K_j^* \subseteq K^*$.\\

So in both cases, $K^*$ is nonmeager in $\el.$ Finally, set $\F' = \ol{\el}.$
 \end{proof}

\begin{cor}
Let $I \subseteq \K(E)$ be a coanalytic ideal with property $(*)$ containing no non-meager sets. Then there exists a closed set $\F \subseteq \K(E)$ such that $\F$ is closed upwards and for any $K \in \K(E)$, \[K \in I \iff K^* \cap \F \textrm{ is meager in } \F.\]
\end{cor}

\begin{proof}
An immediate consequence of Theorem~\ref{thm:sschar} and Theorem~\ref{thm:upclosed}.
\end{proof}

\end{document}